\documentclass[11pt]{article}

\usepackage[utf8]{inputenc}
\usepackage{amsmath}
\usepackage{amsfonts}
\usepackage{amssymb}
\usepackage{amsthm}
\usepackage{bbm}
\usepackage{color, comment}
\usepackage[numbers]{natbib}
\usepackage[colorlinks,citecolor=blue]{hyperref}
\usepackage{tikz}
\usepackage[a4paper,textwidth=14cm]{geometry}
\usepackage{subcaption}

\newcommand{\eps}{\varepsilon}
\renewcommand{\H}{\mathcal{H}}
\newcommand{\N}{\mathbb{N}}
\newcommand{\R}{\mathbb{R}}
\newcommand*{\genbf}[1]{\ifmmode\mathbf{#1}\else\textbf{#1}\fi}

\newcommand{\dist}{\mathrm{dist}}

\newcommand{\1}{\mathbbm{1}}
\renewcommand{\L}{\mathcal{L}}

\newtheorem{theorem}{Theorem}[section]
\newtheorem{lemma}[theorem]{Lemma}
\newtheorem{proposition}[theorem]{Proposition}
\newtheorem{corollary}[theorem]{Corollary}

\theoremstyle{definition}
\newtheorem{remark}[theorem]{Remark}

\title{Geometric criteria for $C^{1,\alpha}$ rectifiability}
\author{Giacomo Del~Nin
\footnote{Mathematics Institute, University of Warwick, Zeeman Building, CV4 7HP Coventry, UK. Email:Giacomo.Del-Nin@warwick.ac.uk}
 \and Kennedy Obinna Idu
 \footnote{Dipartimento di Matematica, Universit\`a di Pisa, Largo Bruno Pontecorvo 5, 56127 Pisa, Italy. Email: idu@mail.dm.unipi.it}
}

\date{}
\begin{document}
\maketitle

\begin{abstract}
\noindent We prove criteria for $\H^k$-rectifiability of subsets of $\R^n$ with $C^{1,\alpha}$ maps, $0<\alpha\leq 1$, in terms of suitable approximate tangent paraboloids. We also provide a version for the case when there is not an a priori tangent plane, measuring on dyadic scales how close the set is to lying in a $k$-plane. We then discuss the relation with similar criteria involving Peter Jones' $\beta$ numbers, in particular proving that a sufficient condition is the boundedness for small $r$ of $r^{-\alpha}\beta_p(x,r)$ for $\H^k$-a.e. $x$ and for any $1\leq p\leq \infty$.
\end{abstract}
\textbf{MSC (2010):} 28A75, 28A78, 26A16.\\
\textbf{Keywords:} rectifiability, approximate tangent paraboloids, beta numbers, H\"{o}lder maps.


\section{Introduction}


A Borel subset $E\subseteq \R^n$ is said to be \emph{countably $\H^k$-rectifiable} (or in short \emph{$\H^k$-rectifiable}) if there exist countably many $k$-dimensional Lipschitz graphs $\Gamma_i$ that cover $E$ up to an $\H^k$-negligible set, that is
$
\H^k\left(E\setminus\bigcup_{i=1}^\infty \Gamma_i\right)=0.
$ 
Here with $k$-dimensional Lipschitz graph we mean the graph of a Lipschitz map over any $k$-dimensional subspace of $\R^n$. Analogously, we say that $E$ is \emph{$\H^k$-rectifiable of class $C^{1,\alpha}$} (or simply \emph{$C^{1,\alpha}$ rectifiable} if the dimension $k$ is clear from the context) if $\Gamma_i$ can be chosen to be graphs of $C^{1,\alpha}$ maps. It is a classical fact  that $\H^k$-rectifiability is equivalent to $C^{1,0}$ $\H^k$-rectifiability \cite[Theorem~15.21]{Mat95}. Here we are interested in obtaining criteria that, starting from some geometric conditions on the set $E$, allow to conclude that $E$ is $C^{1,\alpha}$ rectifiable. 

The main geometric objects used throughout the paper are the \emph{$\alpha$-paraboloids}: for any linear $k$-plane $V$ in $\R^n$, any number $\lambda>0$ and any $x\in \R^n$, following \cite{AnzSer} we define
\begin{equation}\label{eq:alphaparab}
Q_\alpha(x,V,\lambda):=\left\{y\in \R^n: |P_{V^\perp} (y-x) |\le \lambda |P_{V}(y-x)|^{1+\alpha}\right\}
\end{equation}
where $P_V$ denotes the orthogonal projection on $V$. 

The following is the first main result. Throughout the paper we tacitly assume for simplicity that all sets appearing in the statements are Borel.
We also define for simplicity $\eps_0(k):=2^k240^{-k-1}$.
\begin{theorem}[$C^{1,\alpha}$-rectifiability from approximate tangent paraboloids]\label{thm:C1alpharectifgen}
Fix $k\in\{1,\ldots,n-1\}$ and $0<\alpha\leq 1$. Consider a subset $E\subset \R^n$ with $\H^k(E)<\infty$ such that for $\H^k$-a.e. $x\in E$ there exists a $k$-plane $V_x$ and $\lambda>0$ such that
\begin{equation}\label{eq:parabhyp}
\limsup_{r\to 0}{\frac{1}{r^k}\H^k\big(E\cap B(x,r)\setminus Q_\alpha(x,V_x,\lambda)\big)}<\eps_0(k).
\end{equation}
Then $E$ is $C^{1,\alpha}$ $\H^k$-rectifiable.
\end{theorem}
The converse of Theorem \ref{thm:C1alpharectifgen} is also true in a slightly stronger form.

\begin{proposition}[Converse]\label{prop:converse}
If $E$ is a $C^{1,\alpha}$ rectifiable set with $\H^k(E)<\infty$ then for $\H^k$-a.e. $x\in E$ there exist a $k$-plane $V_x$ and $\lambda>0$ such that
\begin{equation}\label{eq:approxcone}
\lim_{r\to 0}{\frac{1}{r^k}\H^k\big(E\cap B(x,r)\setminus Q_\alpha(x,V_x,\lambda)\big)}=0.
\end{equation}
In the case when $\alpha=1$ we can take any $\lambda>0$ arbitrarily small.
\end{proposition}
We call a paraboloid satisfying \eqref{eq:approxcone} an \emph{approximate tangent paraboloid} of $E$ at $x$.\\
The second main result concerns a more general version of Theorem \ref{thm:C1alpharectifgen}, where roughly speaking we allow the plane $V$ to depend not only on $x$ but also on the scale $r$, and moreover we don't require that it passes through $x$. The basic geometric objects we consider are $\eta$-neighbourhoods of $k$-planes, which we call cylinders, that is sets defined by
\[
B(V, \eta):=\{y\in \R^n: \mathrm{dist}(y,V)<\eta\}.
\]
In particular we will consider sets of the form $B(V, \lambda r^{1+\alpha})\cap B(x,r)$ where $V$ is an affine $k$-plane and $\lambda,r>0$, $0<\alpha\leq 1$.

\begin{theorem}[$C^{1,\alpha}$-rectifiability from approximate tangent cylinders]\label{thm:C1alpharot}
Fix $k\in\{1,\ldots,n-1\}$ and $0<\alpha\leq 1$. Consider a subset $E\subset \R^n$ with $\H^k(E)<+\infty$ such that for $\H^k$-a.e. $x\in E$ the following conditions hold:
\begin{equation}\label{eq:ldcyl}
\Theta_*^k(E,x)>0
\end{equation}
and there exist $\lambda>0$ and for every $r>0$ an affine $k$-plane $V_{x,r}$ such that
\begin{equation}\label{eq:cylhyp}
\limsup_{r\to 0}{\frac{1}{r^k}\H^k\big(E\cap B(x,r)\setminus B(V_{x,r}, \lambda r^{1+\alpha})\big)}<(1-2^{-k})\eps_0(k)\Theta_*^k(E,x).
\end{equation}
Then $E$ is $C^{1,\alpha}$ $\H^k$-rectifiable. 
\end{theorem}

\begin{remark}
\begin{itemize}
    \item[(i)] Observe that if $V_{x,r}$ does not depend on $r$, requiring \eqref{eq:parabhyp} or \eqref{eq:cylhyp} is essentially equivalent up to considering different values for $\lambda$ (see Lemma \ref{lemma:parabcyl}). The real difference in the latter case is the possible dependence of $V_{x,r}$ on $r$. We will refer to the first case as \emph{fixed planes}, because they do not change from scale to scale, while to the second as \emph{rotating planes}, because a priori they could change from scale to scale.
    \item[(ii)] This time the lower density assumption \eqref{eq:ldcyl} is necessary in the proof and it is not clear whether we can remove it.
    \item[(iii)] All the results are stated in terms of a set $E$, that is for measures of the form $\H^k\llcorner E$, but the same results hold for Radon measures $\mu$ satisfying $0<\Theta^{*k}(\mu,x)<\infty$ for $\mu$-a.e. $x$ (and $\Theta_*^k(\mu,x)>0$ in Theorem \ref{thm:C1alpharot}). Indeed by standard differentiation results for measures the restriction of $\mu$ to the set $\{x: M^{-1}\leq\Theta^{*k}(\mu,x)\leq M\}$ is equivalent, up to constants, to $\H^k$ restricted to the same set, and by locality the density of a subset remains the same $\H^k$-almost everywhere. The density assumptions from above (or from below in Theorem \ref{thm:C1alpharot}) are all we need to make the proof work. The only difference in this case would be in the explicit value of $\eps_0(k)$.
    \item[(iv)] Conditions \eqref{eq:parabhyp} and \eqref{eq:cylhyp} could actually be required to hold only along a geometric sequence of radii $r_j=r_0\rho^j$, with $0<\rho<1$, thanks to the properties of $\beta$ numbers (see Lemma \ref{lemma:growth}).
\end{itemize}
\end{remark}

To prove Theorem \ref{thm:C1alpharectifgen} and Theorem \ref{thm:C1alpharot} we will actually prove some more quantitative statements. We refer to Proposition \ref{thm:C1alpharectifquant} and Proposition \ref{thm:C1alpharectifCylindersquant} for the precise statements, but since they are a bit technical we limit to state here a simplified version in the case of an Ahlfors-David regular set, that is a compact set $E\subset \R^n$ such that $\delta r^k\leq \H^k(E\cap B(x,r))\leq M r^k$ for every $r\leq \mathrm{diam}(E)$ and for two fixed constants $\delta,M>0$.

\begin{proposition}[Quantitative statement]
Fix $k\in\{1,\ldots,n-1\}$ and $0<\alpha\leq 1$. Fix moreover $\lambda,\delta,M>0$ and consider a set $E\subset \R^n$ such that
\begin{equation}
    \delta r^k\leq \H^k(E\cap B(x,r))\leq Mr^k\quad\text{for every $0<r\leq \mathrm{diam}(E)$}.
\end{equation}
Suppose moreover that for every $x\in E$ there exists a $k$-plane $V_x\in G(n,k)$ such that 
\begin{equation}
\H^k(E\cap B(x,r)\setminus Q_\alpha(x,V_x,\lambda))< \eps_0(k)\delta r^k\quad\text{for every $0<r\leq \mathrm{diam}(E)$}
\end{equation}
where $\eps_0(k)$ is the same as in Theorem \ref{thm:C1alpharectifgen}.
Suppose in addition that
\[
\mathrm{diam}(E)\leq r_1=(4\lambda(2+4^{1+\alpha})+8C(n,\delta,M,\lambda))^{-1/\alpha},
\]
where $C(n,\delta,M,\lambda)$ is the constant of Lemma \ref{lemma:key}(ii).
Then $E$ can be covered by one $k$-dimensional graph of a $C^{1,\alpha}$ map.
\end{proposition}

As a corollary we recover, under some weaker assumptions, a result proved by Ghinassi \cite[Theorem~I]{Ghi}. To state the corollary let us recall (a version of) the $\beta$ numbers, introduced first by Peter Jones in \cite{Jon}:
\begin{align}
\beta_p(x,r)^p&:=\inf_V \frac{1}{r^k}\!\!\int\limits_{E\cap B(x,r)}\!\! \left(\frac{\mathrm{dist}(y,V)}{r}\right)^p d\H^k(y)\quad\text{if $1\leq p<\infty$}\label{eq:beta}\\
\beta_\infty(x,r)&:=\inf_V\H^k\underset{E\cap B(x,r)}{\mathrm{-esssup}} \frac{\mathrm{dist}(y,V)}{r}\nonumber
\end{align}
where the infimum is taken over all affine $k$-planes $V$ (not necessarily containing $x$).

\begin{corollary}[$C^{1,\alpha}$-rectifiability from a bound on $\beta$ numbers]\label{cor:ghi}
Fix $k\in \{1,\ldots, n-1\}$, $0<\alpha\leq 1$. Consider $E\subset \R^n$ with $\H^k(E) < \infty$. Suppose that for $\H^k$-a.e $x\in E$ there exists $1\le p\le\infty$ such that 
\begin{equation}\label{eq:betap}
\limsup_{r\to 0}r^{-\alpha}\beta_p(x,r)<\infty .
\end{equation}
Then $E$ is $C^{1,\alpha}$ $\H^k$-rectifiable.
\end{corollary}

The proof of the above corollary of Theorem \ref{thm:C1alpharot} relies on Theorem \ref{thm:QuantRectifGen} below, which is used to obtain the positive lower density assumption \eqref{eq:ldcyl}. We could avoid using it if we assume directly that $\Theta_*^k(E,x)>0$ for $\H^k$-a.e. $x\in E$.

\begin{remark}[$C^{1,\omega}$-rectifiability]\label{rmk:intromoduli}
We stated the results for $C^{1,\alpha}$ rectifiability, but actually with a totally analogous proof (see Remark \ref{rmk:moduli}) we can prove the following result regarding more general moduli of continuity: consider a positive sequence $(\lambda_j)_{j\in \N}$ such that $\sum_j \lambda_j <\infty$ and consider the tail sequence given by $\omega_m:=\sum_{j=m}^\infty \lambda_j$. Fix a geometric sequence $r_j=r_0\rho^j$, $0<\rho<1$, and define for any $r>0$
\[
\lambda(r)=\lambda_{j(r)},\qquad \omega(r)=\omega_{j(r)}\qquad\text{where\, $r_{j(r)+1}\leq r< r_{j(r)}$}
\]
If we replace the left hand side in \eqref{eq:parabhyp} or in \eqref{eq:cylhyp} with
\[
\limsup_{r\to 0}\frac{1}{r^k}\H^k(E\cap B(x,r)\setminus B(V_{x,r}, \lambda(r) r))
\]
then $E$ is $C^{1,\omega}$ $\H^k$-rectifiable, meaning that up to an $\H^k$-negligible set it can be covered by countably many graphs of $C^1$ maps whose derivative has modulus of continuity $\omega$. Theorem \ref{thm:C1alpharot} corresponds to the choice $\lambda_j=r_j^\alpha$, or equivalently up to constants to $\lambda(r)=r^\alpha$, which indeed gives $\omega(r)=Cr^\alpha$. 
\end{remark}

Let us now compare the above criteria with previously known results.
First of all we compare them with the following classical rectifiability criterion. Given $s>0$ and $x\in \R^n$ we define the cone 
\[
X(x,V,s)=\big\{y\in \R^n : |P_{V^\perp}(y-x)|\leq s |P_V (y-x)|\big\}.
\]
The following result is a consequence of \cite[Corollary~15.16]{Mat95} (we underline the slightly different definition of the cones $X(x,V,s)$ in the reference above).

\begin{theorem}[Rectifiability from approximate tangent cones]\label{thm:rectifgen}
Given a set $E\subseteq\R^n$ with $\H^k(E)<+\infty$, suppose that for $\H^k$-a.e. point $x\in E$ there exist a $k$-plane $V$ and $s>0$ such that
\begin{equation}\label{eq:Xrk}
\limsup_{r\to 0}\frac{1}{(2r)^k}\H^k\big(E\cap B(x,r)\setminus X(x,V,s)\big)< \frac{1}{240^{k+1}}\left(\frac{1}{\sqrt{1+s^2}}\right)^k.
\end{equation}
Then $E$ is $\H^k$-rectifiable.
\end{theorem}


 We can view Theorem \ref{thm:C1alpharectifgen} as a direct analogue of Theorem \ref{thm:rectifgen} obtained by replacing cones with paraboloids. Observe that under the assumptions of Theorem~\ref{thm:C1alpharectifgen} the set $E$ is $\H^k$-rectifiable, since we can take $s$ small enough so that \eqref{eq:Xrk} is satisfied, and thus Theorem \ref{thm:rectifgen} applies.
 
Theorem \ref{thm:C1alpharot} is instead related to \cite[Theorem 16.2]{Mat95}. One difference (besides the
fact that only planes through the point are considered) is that the author asks as an assumption
that not only the set $E$ is close to some $k$-plane on every scale \cite[{(15.9)}]{Mat95}, but
also viceversa, that in the same scale the $k$-plane is close to the set $E$ \cite[{(15.8)}]{Mat95}, which is much more than \eqref{eq:ldcyl}. 
A stronger assumption like this one is really necessary because there are
examples of sets which satisfy only \cite[{(15.9)}]{Mat95} and also satisfy \eqref{eq:ldcyl} but are not rectifiable (see \cite[Chapter 20]{DS}, choosing $\alpha_n\to 0$ but $\sum \alpha_n^2=+\infty$).
Instead in our case the precise rate of shrinking of the set around the
planes that we assume (of order $r^{1+\alpha}$) is strong enough to imply
rectifiability just assuming positive lower density. 
 
Anzellotti and Serapioni \cite{AnzSer} already considered criteria of $C^{1,\alpha}$ rectifiability in terms of approximate tangent parabolic sets. However their approach is a bit different. In fact they prove that the $C^{1,\alpha}$ 
rectifiability of a set is equivalent to a certain condition involving non-homogeneous
blow-ups, together with the (approximate) $\alpha$-H\"{o}lder continuity of the tangent planes. One feature of Theorem \ref{thm:C1alpharectifgen}, instead, is that we do not require any a priori H\"{o}lder continuity of the planes $V_x$ since we deduce it from the geometric condition \eqref{eq:parabhyp} and natural
density bounds for the Hausdorff measure. Essentially if the planes were not H\"{o}lder continuous then the measure would locally concentrate too much around $(k-1)$-dimensional subspaces, which is prevented by the density estimates for the Hausdorff measure, see Lemmas \ref{lemma:tube} and \ref{lemma:key}. Moreover the non-homogeneous blow-ups prevent the possibility of obtaining the converse result \cite[Remark after Theorem 3.5]{AnzSer}, which instead we obtain in Proposition \ref{prop:converse}.  Finally we give a quantitative
result where we cover the entirety of the set with one single $C^{1,\alpha}$ graph,
without an additional negligible set.

A more recent progress in generalising the result of Anzellotti and Serapioni to $C^{k,\alpha}$ rectifiability is due to Santilli \cite{San17} who develops a differentiability notion of higher order for subsets of the Euclidean space that allows for the characterization of higher order rectifiable sets. We remark that in the case of $C^{1,\alpha}$ rectifiability, $0<\alpha\le 1$, Theorem \ref{thm:C1alpharectifgen} with $\eps_0(k)=0$ corresponds to \cite[Theorem 5.4]{San17} and Proposition \ref{prop:converse} to \cite[Theorem 3.23]{San17}.
Other results which employ various techniques in special cases of higher order rectifiability can be found in \cite{Del07, Del08} for the case of curves, and also in \cite{Alb94} for the case of singular sets of convex functions.

We now compare Theorem \ref{thm:C1alpharot} with similar criteria involving $\beta$ numbers defined in \eqref{eq:beta}. We could roughly summarize the difference by saying that instead of requiring an $L^2$-closeness to some $k$-plane at every scale we just require closeness in measure. 
The following is a characterization of $\H^k$-rectifiability using $\beta$ numbers due to Azzam and Tolsa (see also \cite{Paj} for a previous result assuming a lower density assumption). The original result is about Radon measures with positive and finite upper density, but we state it in the case of sets, i.e. for measures of the form $\H^k\llcorner E$ for some set $E$.
\begin{theorem}[\citep{AzzTol},\cite{Tol}]\label{thm:QuantRectifGen} Let  $E \subset \mathbb{R}^{n}$ be an $\mathcal{H}^{k}$ -measurable set with $\mathcal{H}^{k}(E)<\infty$. The set $E$ is $\H^k$-rectifiable if and only if
\[
\int_{0}^{1} \beta_2(x, r)^{2} \frac{d r}{r}<\infty \quad \text { for } \mathcal{H}^{k}\text {-a.e. } x \in E.
\]
\end{theorem}
This essentially captures rectifiability by measuring in a scale-invariant way, using an $L^2$ gauge, how close a set is to lying in a $k$-plane. Ghinassi \cite{Ghi} proved a sufficient condition for a Radon measure on $\R^n$ to  be $C^{1,\alpha}$-rectifiable, $0<\alpha \le 1$, using a similar condition which involves the $\beta$ numbers.
\begin{theorem}[{\cite[Theorem I]{Ghi}}]\label{thm:C1aRectGhi}
Let \(E \subseteq \mathbb{R}^{n}\) such that \(0<\Theta^{k *}(E, x)<\infty\) for \(\H^{k}\) a.e. \(x \in E\), and let \(\alpha \in(0,1) .\) Assume that for $\H^k$-a.e. \(x \in E\)
\begin{equation}\label{GhCond}
\int_0^1 \frac{\beta_\infty(x,r)^2}{r^{2\alpha}} \frac{dr}{r}<\infty.
\end{equation}
Then \(E\) is \(C^{1, \alpha}\) k-rectifiable.
\end{theorem}
 
 We recover this result from Corollary \ref{cor:ghi}: condition \eqref{GhCond} (together with Lemma \ref{lemma:growth}) implies \eqref{eq:betap}, and therefore Corollary \ref{cor:ghi} applies.
A remarkable feature of Theorems \ref{thm:QuantRectifGen} and \ref{thm:C1aRectGhi} is that there is no a priori condition
on the choice of $k$-planes, which can vary from scale to scale. Corollary \ref{cor:ghi} weakens a bit the assumptions of Theorem
\ref{thm:C1aRectGhi}, requiring only a bound on $r^{-\alpha}\beta_p$ instead of an integrability condition. We remark that in \cite{Ghi} Ghinassi was also interested in obtaining a biLipschitz
parametrization of the set which is $C^{1,\alpha}$ in both ways, under some Reifenberg-flatness assumptions, building on previous work by David and Toro \cite{DavTor, Tor}.
With the present work we intend to show how a more “classical” approach,
that follows the lines of similar criteria for standard $C^1$ rectifiability just replacing
approximate tangent cones with paraboloids, is capable of giving a relatively simple
proof of the $C^{1,\alpha}$ rectifiability.

We also remark that the key property that makes the proof work is the summability of $r_j^\alpha$, or more generally of $(\lambda_j)_{j\in\N}$ in Remark \ref{rmk:intromoduli} (equivalently, the integrability on $(0,1)$ of $\lambda(r)$ in $\tfrac{dr}{r}$), which is for instance implied by the summability of the $\beta_\infty(x,r_j)$ and which we exploit to obtain the convergence of $V_{x,r}$ to the tangent plane $V_x$. This is in contrast with Theorem \ref{thm:QuantRectifGen} where instead just the \emph{square} summability of the $\beta$ numbers is required, but at the cost of a more complex proof. A similar difference of assumptions between summability and square summability can be seen between \cite{ADT} and \cite{ATT}, this time involving the $\alpha$ numbers. Since we want to obtain the $C^{1,\alpha}$ rectifiability anyway we take full advantage of the summability of $r_j^\alpha$.

We finally mention that in more recent years there has been a growing interest in conditions for higher order rectifiability related to Menger-type curvatures \cite{Kol, GhiGoe}, in the spirit of the original result by L\'{e}ger \cite{Leg} and its generalization to higher dimension \cite{Meu}. The connection established in \cite{LW09, LW11} by Lerman and Whitehouse reveals the similarities and transference of methodology in the literature on criteria involving the Menger-type curvatures and those of the beta numbers. This further suggests investigating weaker assumptions as we have explored with regard to the beta numbers for the results on Menger-type curvatures.\\

\subsection*{Acknowledgements.} We are deeply indebted to Giovanni Alberti for motivating the problem and for many helpful discussions.
This project has received funding from the European Research Council (ERC) under the European Union's Horizon 2020 research and innovation programme under grant agreement No 757254 (SINGULARITY).

\section{Preliminaries}
We denote the Euclidean norm in $\R^n$ by $|\cdot|$. We denote by $G(n,k)$ the Grassmannian of linear $k$-planes in $\R^n$, and given $V\in G(n,k)$ we define $P_V$ as the orthogonal projection on $V$. In the following we often identify $\R^k$ with the span of the first $k$ standard basis vectors in $\R^n$, and $\R^{n-k}$ with the span of the last $n-k$, and we write a vector in $\R^n$ as $y=(y',y'')$ with $y'\in\R^k$, $y''\in \R^{n-k}$. On $G(n,k)$ we consider the metric 
\[
d(V,W):=\|P_V-P_W\|=\sup_{\substack{x\in \R^n\\ |x|=1}} |(P_V-P_W)x|.
\]
We will also consider the following equivalent expression for the distance $d$:
\begin{equation}\label{eq:dist}
d(V,W)=\sup_{\substack{v\in V\\ |v|=1}}\mathrm{dist}(v,W)=\sup_{\substack{v\in V\\ |v|=1}} |P_{W^\perp}v|
\end{equation}
where $\mathrm{dist}(v,W)=\inf_{w\in W} |v-w|$.
For a proof of the equality \eqref{eq:dist} see \cite[Section~34]{AG93}. We note that by compactness all suprema are attained. We also note the following: if $V$ is a $k$-plane which is the graph of a linear function $L:\R^k\to \R^{n-k}$, that is $V=\{(y',y'')\in \R^n:y''=Ly'\}$, then 
\begin{equation}\label{eq:distequiv}
d(V,\R^k)\leq \|L\|\leq \frac{d(V,\R^k)}{\sqrt{1-d(V,\R^k)^2}}.
\end{equation}
Indeed
\[
d(V,\R^k)=\sup_{\substack{y'\in \R^k\\ |y'|=1}} \mathrm{dist}(y',V)\leq \sup_{\substack{y'\in \R^k\\ |y'|=1}} |Ly'|=\|L\|.
\]
We also consider the space $A(n,k)$ of all affine $k$-planes in $\R^n$. Given $V\in A(n,k)$ we can always write in a unique way, $V\in A(n,k)$ as, $V=\tilde V+a$ with $\tilde V\in G(n,k)$ and $a\in \tilde V^\perp$. We define the (pseudo)distance $d(V,W):=d(\tilde V,\tilde W)$ where $\tilde V,\tilde W\in G(n,k)$ are the linear $k$-planes associated to $V,W$ as above.

Given $V\in A(n,k)$ and $\eta>0$ we define the $\eta$-neighbourhood $B(V, \eta):=\{x\in \R^n: \mathrm{dist}(x,V)<\eta\}$.

We denote by $\L^n$ the Lebesgue measure in $\R^n$ and following \cite{Mat95} we define the $k$-dimensional Hausdorff measure of a set $E\subset \R^n$ by $\H^k(E)=\lim\limits_{\delta\to 0}\H^k_\delta(E)$ where
\[
\H^k_\delta(E)=\inf\left\{\sum_i \mathrm{diam}(E_i)^k:\,E\subset\bigcup_i E_i,\, \mathrm{diam}(E_i)\leq \delta\right\}.
\]
Note in particular that there is no normalizing factor so that in fact in $\R^n$ we have $\H^n(B(x,r))=(2r)^n$.

Given a subset $E\subseteq \R^n$ we define the upper and lower $k$-dimensional Hausdorff densities respectively as
\begin{align*}
\Theta^{*k}(E,x)&=\limsup_{r\to 0} \,(2r)^{-k}\H^k(E\cap B(x,r))\\
\Theta_*^{k}(E,x)&=\liminf_{r\to 0} (2r)^{-k}\H^k(E\cap B(x,r))
\end{align*}
and when they coincide we denote the common value by $\Theta^k(E,x)$. By \cite[Theorem~6.2]{Mat95} if $\H^k(E)<\infty$ then for $\H^k$-a.e. $x\in E$
\begin{equation}\label{eq:Mat6.2}
2^{-k}\leq\Theta^{*k}(E,x)\leq 1.
\end{equation}
Moreover if $E$ is $\H^k$-rectifiable then for $\H^k$-a.e. $x\in E$ it holds that $\Theta^k(E,x)=1$ \cite[Theorem~16.2]{Mat95}.

Throughout the paper we will often refer to a geometric sequence of radii defined by $r_j=r_0\rho^j$, where $r_0>0$ and $0<\rho <1$ are fixed.

We will keep track of the constants in order to be able to prove some quantitative statements, but we will not try to optimize them in any way.

The remaining of this section could be skipped at a first reading and used only when referenced later. The only notion required for the sequel is that of slanted paraboloids \eqref{eq:alphaparabL}.

\subsection{Intersection of cylinders around planes}

\begin{lemma}\label{lemma:tube}
Consider two linear $k$-planes $V,W\in G(n,k)$ and set $\theta=d(V,W)$. Then there exists a $(k-1)$-plane $Z$ such that for any positive number $\eta$ we have
\[
B(V, \eta)\cap B(W, \eta)\subseteq B\left(Z, \frac{2n\eta}{\theta}\right).
\]
\end{lemma}

\begin{proof} We first observe that there exists $e\in W^\perp$ with $|e|=1$ such that $|P_V e|= \theta$. Indeed we have $d(V,W)=\|P_V-P_W\|=\|P_{W^\perp}-P_{V^\perp}\|=d(V^\perp,W^\perp)$ and then by \eqref{eq:dist} we obtain
\[
\theta=d(V,W)=d(V^\perp,W^\perp)=\sup_{\substack{e\in W^\perp \\|e|=1}} |P_V e|.
\]
Now we consider any orthonormal basis $e_{k+1},\ldots,e_n$ of $V^\perp$ and we define $Z:=\mathrm{span}\{e,e_{k+1},\ldots,e_n\}^\perp$. Then $\mathrm{dim}\, Z=k-1$ and for any $x\in B(V, \eta)\cap B(W,\eta)$ we have
\[
\begin{cases}
|x\cdot e_i|\leq \eta & \text{ for $i=k+1,\ldots, n$}\\
|x\cdot e|\leq \eta
\end{cases}.
\]
We now set $e':=\tfrac{P_V e}{|P_V e|}$ and consider the orthonormal basis $\{e',e_{k+1},\ldots, e_n\}$ of $Z^\perp$. Then for any $x\in B(V, \eta)\cap B(W, \eta)$ by triangle inequality we have
\[
|x\cdot e'|=\frac{1}{|P_V e|} |x\cdot P_V e|=\frac{1}{\theta} \left|x\cdot \left(e-\sum_{i=k+1}^n (e\cdot e_i)e_i\right)\right|\leq \frac{1}{\theta}(n-k+1)\eta
\]
and finally
\[
|P_{Z^\perp}x|\leq|x\cdot e'|+\sum_{i=k+1}^n |x\cdot e_i|\leq \frac{1}{\theta}(2n-2k+1)\eta\leq \frac{2n}{\theta}\eta.
\]
This concludes the proof.
\end{proof}

\subsection{\texorpdfstring{$\alpha$}{alpha}-paraboloids}
For technical reasons we now introduce a variant of the $\alpha$-paraboloids defined in \eqref{eq:alphaparab} which takes into account the possibility of slanting the paraboloid with a linear map. Given $V\in G(n,k)$, a linear map $L:V\rightarrow V^\perp$ and $\lambda>0$ we define
\begin{equation}\label{eq:alphaparabL}
Q_\alpha^L(V, \lambda)=\{y\in\R^n:| P_{V^\perp}y - L(P_{V}y) |\le \lambda |P_{V}y|^{1+\alpha}\}
\end{equation}
and then for any $x\in \R^n$ we set $Q_\alpha^L(x,V,\lambda)=x+Q_\alpha^L(V, \lambda)$. These are paraboloids with a linear perturbation parametrized over $V$. If the linear map $L$ is constantly zero then we recover the definition \eqref{eq:alphaparab}.

We now proceed to show that these sets and the $\alpha$-paraboloids introduced in \eqref{eq:alphaparab} are, in fact, locally equivalent. Since this property is invariant under rigid motions, without loss of generality in the proof we consider the case where the $k$-plane in \eqref{eq:alphaparabL} is $\R^k$.

\begin{lemma}[Local equivalence of paraboloids]\label{lemma:parabequiv}
Let $V\in G(n,k)$ be given by the graph of a linear map $L:\R^k\to\R^{n-k}$ and suppose that $\|L\|\leq\frac12$. Then for any given $\lambda > 0$ we have that
\[
Q_\alpha(x,V,\lambda)\cap B(x,r_\lambda)\subseteq Q_\alpha^L(x,\R^k,\lambda')
\]
where $r_\lambda=(4\lambda)^{-1/\alpha}$ and $\lambda'=6\cdot 4^\alpha \lambda$.
\end{lemma}

\begin{proof}
We can suppose without loss of generality that $x=0$. Take any $y\in Q_\alpha(0,V,\lambda)\cap B(0,r_\lambda)$. We want to show that 
\[
|P_{\R^{n-k}}y-LP_{\R^k}y|\leq \lambda'|P_{\R^k}y|^{1+\alpha}. 
\]
Since $y\in Q_\alpha(0,V,\lambda)\cap B(0,r_\lambda)$ then
\[
|P_{V^\perp}y|\leq \lambda |P_V y|^{1+\alpha}\leq \lambda |y|^{1+\alpha}\leq \frac14 |y|
\]
which implies $|P_V y|\geq |y|-|P_{V^\perp}y|\geq \tfrac34 |y|$.
Using that $d(V,\R^k)\leq \|L\|\leq \tfrac12$ it follows that  
\begin{equation}\label{eq:PV1}
|P_{\R^k}y|\geq|P_V y|-|(P_{\R^k}-P_V)y|\geq \frac14 |y|\geq \frac14 |P_V y|.
\end{equation}
On the other hand since $P_V y\in V$ we have that $P_{\R^{n-k}}P_V y-LP_{\R^k}P_V y=0$ and thus
\begin{equation}\label{eq:PL}
    |P_{\R^{n-k}}y-LP_{\R^k}y|=|P_{\R^{n-k}}P_{V^\perp}y-LP_{\R^k}P_{V^
\perp}y|\leq (1+\|L\|)|P_{V^\perp}y|\leq \frac32 |P_{V^\perp}y|.
\end{equation}
Putting together \eqref{eq:PV1} and \eqref{eq:PL} we finally obtain
\[
|P_{\R^{n-k}}y-LP_{\R^k}y|\leq \frac32 |P_{V^\perp}y|\leq \frac32\lambda |P_V y|^{1+\alpha}\leq 6\cdot 4^\alpha \lambda |P_{\R^k} y|^{1+\alpha}.
\]
\end{proof}

\begin{lemma}[Relation between cylinders and paraboloids]\label{lemma:parabcyl}
Fix $V\in G(n,k)$ and $r_0$ and suppose that for every $r<r_0$, 
\[\H^k(E\cap B(x,r)\setminus B(V, \lambda  r^{1+\alpha}))\le \varepsilon r^k.
\]
Then for every $r<r_0$:
\[
\H^k(E\cap B(x,r)\setminus Q_\alpha(x,V,\lambda'))\le \frac{\varepsilon}{1-2^{-k}} r^k
\]
where $\lambda'=4^{1+\alpha}\lambda$.
\end{lemma}

\begin{proof}
 A computation gives that for sufficiently small $r$ we have
 \[
 \left(B(x,r)\setminus B(x,\tfrac{r}{2})\right)\setminus Q_\alpha(x,V,\lambda')\subset B(x,r)\setminus B(V, \lambda r^{1+\alpha})
 \]
  As a consequence we obtain
 \begin{align*}
 B(x,r)\setminus Q_\alpha(x,V,\lambda')&= \bigcup_j \left(B(x,\tfrac{r}{2^j})\setminus B(x,\tfrac{r}{2^{j+1}})\right)\setminus Q_\alpha(x,V,\lambda')\\
 &\subset\bigcup_j B(x,\tfrac{r}{2^j})\setminus B\left(V, \lambda \left(\tfrac{r}{2^j}\right)^{1+\alpha}\right).
 \end{align*}
 Using the assumption we obtain that
 \[
 \H^k\left(E\cap B(x,r)\setminus Q_\alpha(x,V,\lambda')\right)\leq \sum_j \eps \left(\frac{r}{2^j}\right)^k=\frac{\eps}{1-2^{-k}} r^k.
 \]
\end{proof}

\subsection{Properties of \texorpdfstring{$\beta$}{beta} numbers}
We provide two lemmas regarding $\beta$ numbers. The first allows to control the values of $\beta(x,r)$ at smaller scales with those at bigger scales. The second shows that it is equivalent to require the summability of $\beta(x,r)$ along a geometric sequence of radii or its integrabiity in $\tfrac{dr}{r}$ near the origin.
\begin{lemma}\label{lemma:growth}
For any $t\in (0,1)$
\[
\beta_2(x,t r)^2\leq \frac{1}{t^{k+2}}\beta_2(x,r)^2
\]
\end{lemma}

\begin{proof}
 We consider a $k$-plane $V_{x,r}$ that realizes the infimum for $\beta_2(x,r)$ in \eqref{eq:beta} and use it as a competitor in the definition of $\beta_2(x,t r)$. Then
 \begin{align*}
 \beta_2(x,t r)^2 &\leq \frac{1}{(t r)^k}\int\limits_{E\cap B(x,t r)}\left(\frac{\dist(y,V_{x,r})}{t r}\right)^2 d\H^k(y)\\
 & \leq \frac{1}{t^{k+2}}\frac{1}{r^k}\int\limits_{E\cap B(x,r)}\left(\frac{\dist(y,V_{x,r})}{r}\right)^2 d\H^k(y)\\
 &=\frac{1}{t^{k+2}}\beta_2(x,r)^2.
 \end{align*}
\end{proof}

\begin{lemma}\label{lemma:fubini}
Fix $0<\rho<1$. Given $r_0\in(\rho,1]$ define $r_j=r_0 \rho^j$.
\begin{itemize}
\item[(i)] If for some $r_0$ we have $\sum_{j=0}^\infty \beta_2(x,r_j)^2<\infty$ then $\int_0^1 \beta_2(x,r)^2\frac{dr}{r}<\infty $.

\item[(ii)] Conversely, if $\int_0^1 \beta_2(x,r)^2\frac{dr}{r}<\infty$ then for every choice of $r_0$ in $(\rho,1]$  we have $
\sum_{j=0}^\infty \beta_2(x,r_j)^2<\infty$.
\end{itemize}
\end{lemma}

\begin{proof}
By Fubini and the change of variables $r=r_0\rho^j$
\begin{align*}
\int_\rho^1 \left(\sum_{j=0}^\infty \beta_2(x,r_0\rho^j)^2\right)dr_0 &=\sum_{j=0}^\infty \int_\rho^1\beta_2(x,r_0\rho^j)^2dr_0\\
  & =\sum_{j=0}^\infty \int_{\rho^{j+1}}^{\rho^j} \beta_2(x,r)^2\frac{1}{\rho^j}dr\\
& = \int_0^1 \beta_2(x,r)^2\underbrace{\left(\sum_{j=0}^\infty \frac{1}{\rho^j}\1_{[\rho^{j+1},\rho^j)}(r)\right)}_{g_\rho(r)}dr.
\end{align*}
Since $ g_\rho(r)\leq \frac{1}{r}$ we obtain that the series converges for $\H^1$-a.e. choice of $r_0$ in $(\rho,1]$, but by Lemma \ref{lemma:growth} once it converges for one choice of $r_0$ it converges for all choices, and thus we obtain $(ii)$. On the other hand we use the fact that for any $t\in (0,1)$ $\beta_2(x,tr)^2\leq \frac{1}{t^{k+2}}\beta_2(x,r)^2$ and that $g_\rho(r)\geq\frac{\rho}{r}$ to obtain $(i)$.
\end{proof}

\begin{lemma}[$\beta_2$ versus $\beta_p$]\label{lemma:betap}
We have the following relations:
\[
\begin{cases}
\beta_2(x,r)\leq 2^{2/p-1}\beta_p(x,r)^{2/p}&\text{ if $1\leq p\leq 2$}\\
\beta_2(x,r)\leq 2^{k/2}\Theta^{*k}(E,x)^{1/2-1/p}\beta_p(x,r) \text{ for $r$ sufficiently small}& \text{ if $2\leq p\leq \infty$}
\end{cases}.
\] 
\end{lemma}

\begin{proof}
The first case follows from the fact that $\dist(y,V_{x,r})\leq 2r$ if $y\in E\cap B(x,r)$ and $V_{x,r}$ is the minimizer for $\beta_p(x,r)$. The second follows from H\"{o}lder inequality: if $2\leq p<\infty$ and $V_{x,r}$ is a minimizer for $\beta_p(x,r)$ then
\begin{align*}
\beta_2(x,r)^2& \leq\frac{1}{r^k}\int_{E\cap B(x,r)}\left(\frac{\dist(y,V_{x,r})}{r}\right)^2d\H^k(y)\\
& \leq \left(\frac{1}{r^k}\int_{E\cap B(x,r)} \left(\frac{\dist(y,V_{x,r})}{r}\right)^pd\H^k(y)\right)^{2/p}\left(\frac{1}{r^k}\H^k\big(E\cap B(x,r)\big)\right)^{1-2/p}
\end{align*}
and the conclusion follows. The case $p=\infty$ is treated in a similar fashion.
\end{proof}

\subsection{Whitney's extension theorem}
In the following we will need a Whitney-type extension theorem. We quote a version that can be found in Stein's book \cite[VI.2.3, Theorem~4]{Stein}, which in the particular case of $C^{1,\alpha}$ extensions reduces to the following simplified statement:

\begin{theorem}[$C^{1,\alpha}$ extension]\label{Lemma:Whitney}
Consider any subset $F\subset\R^k$ and a function $f:F\to \R^d$. Then the following are equivalent:
\begin{itemize}
    \item[(a)] $f$ admits a bounded $C^{1,\alpha}$ extension $\tilde f:\R^k\to \R^d$ with bounded derivatives;
    \item[(b)] there exist $M>0$ and for every $x\in F$ linear maps $L_x:\R^k\to\R^d$ such that for every $x,y\in F$
    \begin{itemize}
    \item[(i)] $|f(x)-f(y)-L_x(y-x)|\leq M|x-y|^{1+\alpha}$;
    \item[(ii)] $\|L_x-L_y\|\leq M|x-y|^\alpha$;
    \item[(iii)] $|f(x)|,\|L_x\|\leq M$.
    \end{itemize}
\end{itemize}
\end{theorem}

\section{Fixed planes}
In this section we prove Theorem \ref{thm:C1alpharectifgen} and Proposition \ref{prop:converse}. We start from the latter.

\begin{proof}[Proof of Proposition \ref{prop:converse}]
Let $E$ be $C^{1,\alpha}$ rectifiable, thus $\H^k\left(E\setminus \bigcup_{i\in \mathbb{N}}\Gamma_i\right)=0$ with $\Gamma_i$ graphs of $C^{1,\alpha}$ maps. In particular $E$ is $\H^k$ rectifiable and for $\H^k$-a.e. $x\in E$ there exists a tangent plane $V_x$ and $\Theta^k(E,x)=1$ . By locality properties of the density, for $\H^k$-a.e. $x\in E_i:=E\cap \Gamma_i$ we have $\Theta^k(E_i,x)=1$ and thus 
\begin{equation}\label{eq:E_i}
    \Theta^k(E\setminus E_i,x)=0 \text{ for $\H^k$-a.e. $x\in E_i$.}
\end{equation}
Now by Lemma \ref{lemma:parabequiv} and Lemma \ref{Lemma:Whitney}, for some $\lambda>0$ we have that $E_i\cap B(x,r)\setminus Q_\alpha(x,V_x,\lambda)=\emptyset$ for every $x\in E_i$ if $r$ is sufficiently small. Putting together the latter fact and \eqref{eq:E_i} we conclude.
\end{proof}

We now pass to the proof of Theorem \ref{thm:C1alpharectifgen}. We will prove the following more quantitative result, from which Theorem \ref{thm:C1alpharectifgen} will directly follow.

\begin{proposition}[Quantitative statement]\label{thm:C1alpharectifquant}
Fix $k\in\{1,\ldots,n-1\}$ and $0<\alpha\leq 1$. Fix moreover $r_0,\lambda,\delta,M>0$ and consider $F'\subseteq F\subseteq \R^n$ such that for every $x\in F$
\begin{equation}\label{eq:upper}
    \H^k(F\cap B(x,r))\leq Mr^k\quad\text{for every $0<r\leq r_0$}.
\end{equation}
Suppose moreover that for every $x\in F'$ 
\begin{equation}\label{eq:ldparab}
\H^k(F\cap B(x,r))\geq\delta r^k \quad\text{for every $0<r\leq r_0$}
\end{equation}
and there exists a $k$-plane $V_x\in G(n,k)$ such that
\begin{equation}\label{eq:parabhypquant}
\H^k(F\cap B(x,r)\setminus Q_\alpha(x,V_x,\lambda))\leq \eps r^k
\end{equation}
with $\eps< \frac{1}{4^k+1}\delta$. If $\mathrm{diam}(F')\leq r_1=(4\lambda(2+4^{1+\alpha})+8C)^{-1/\alpha}$ then $F'$ can be covered by one $k$-dimensional graph of a $C^{1,\alpha}$ map, where $C=C(n,\delta,M,\lambda)$ is the constant of Lemma \ref{lemma:key}(ii).
\end{proposition}
 A version of the above for the case of rotating planes is also true:
 
 \begin{proposition}[Quantitative statement for rotating planes]\label{thm:C1alpharectifCylindersquant}
Fix $k\in\{1,\ldots,n-1\}$ and $0<\alpha\leq 1$. Fix moreover $r_0,\lambda,\delta,M>0$ and consider $F'\subseteq F\subseteq \R^n$ such that for every $x\in F$ and for every $0<r<r_0$ we have
 \begin{equation}\label{eq:upperCyl}
    \H^k(F\cap B(x,r))\leq Mr^k.
\end{equation}
Suppose moreover that for every $x\in F'$ and for every $0<r<r_0$ we have
\begin{equation}\label{eq:ldparabCyl}
\H^k(F\cap B(x,r))\geq\delta r^k,
\end{equation}
and there exists a $k$-plane $V_{x,r}$ such that
\begin{equation}\label{eq:parabhypCyl}
\H^k(F\cap B(x,r)\setminus B(V_{x,r}, \lambda r^{1+\alpha}))\leq \eps r^k
\end{equation}
with $\eps< \frac{1-2^{-k}}{4^k+1}\delta$. If $\mathrm{diam}(F')\leq r_1=(4\lambda(2+4^{1+\alpha})+8C)^{-1/\alpha}$ then $F'$ can be covered by one $k$-dimensional graph of a $C^{1,\alpha}$ map, where $C=C(n,\delta,M,\lambda)$ is the constant of Lemma \ref{lemma:key}(ii).
\end{proposition}

We start with a geometric lemma used to cover the totality of a set with the graph of a $C^{1,\alpha}$ map. It is a geometric version of Whitney's extension theorem.
\begin{lemma}[Geometric lemma]\label{lemma:geom}
Fix $\lambda,C>0$ and consider a set $E\subset \R^n$ with $\mathrm{diam}(E)\leq \min\{(4C)^{-1/\alpha},(4\lambda)^{-1/\alpha}\}$ such that for every $x\in E$ there exists a $k$-plane $V_x$ such that 
\[
E\setminus Q_\alpha(x,V_x,\lambda)=\emptyset
\]
and moreover $d(V_x,V_y)\le C|x-y|^\alpha$ for every $x,y\in E$. Then $E$ can be covered by one $C^{1,\alpha}$ $k$-dimensional graph.
\end{lemma}

\begin{proof}
Thanks to the assumption on $\mathrm{diam}(E)$, for any fixed $x_0\in E$ we have that $d(V_{x_0},V_x)\leq C|x-x_0|^\alpha\leq \tfrac14$, hence $V_x$ can be parametrized as a graph of a linear map $L_x$ over $V_{x_0}$ with $\|L_x\|\leq\tfrac12$ by \eqref{eq:distequiv}. We now assume without loss of generality that $V_{x_0}=\R^k$. By Lemma \ref{lemma:parabequiv} for every $x\in E$
\begin{equation}\label{eq:EQalpha}
E\setminus Q_\alpha^{L_x}(x,\R^k,6\cdot 4^\alpha\lambda)=\emptyset,
\end{equation}

We can thus apply Whitney's extension theorem given by Lemma \ref{Lemma:Whitney}: clearly \eqref{eq:EQalpha} implies that $E$ is a graph of a function $f:\R^k\to \R^{n-k} $ since the projection $P_{\R^k}$ is injective on $E$. Assumption $(iii)$ is trivially satisfied. Assumption $(i)$ follows from \eqref{eq:EQalpha} and the definition of slanted paraboloids \eqref{eq:alphaparabL}, while $(ii)$ holds by assumption.
\end{proof}

The following is the key lemma where we essentially prove that the H\"{o}lder continuity of the tangent planes is a consequence of the paraboloid condition \eqref{eq:parabhyp} together with the positive lower density and the finite upper density. 

\begin{remark}
In the statement and in the proof below, given a point $x$ and parameters $r,\lambda>0$ and an affine $k$-plane $V_x$ we use the simplified notation $C_\alpha^r(x)$ to denote $B(V_x, \lambda r^{1+\alpha})$, where $\lambda>0$ is fixed.
\end{remark}


\begin{lemma}\label{lemma:key}
Consider a set $F\subset\R^n$ and fix $M,\lambda,\delta, r>0$. Suppose that for every $z\in F$ and for every $0<s\leq 5r$ we have 
\begin{equation}\label{eq:upperhyp}
\H^k(F\cap B(z,s))\leq Ms^k.
\end{equation}
Consider any two points $x,y$ such that $|x-y|\leq  r$ and two $k$-planes $V_x,V_y\in A(n,k)$ satisfying
\begin{equation}\label{eq:lowerbound}
\begin{cases}
\H^k(F\cap B(x, r))\geq \delta  r^k\\
\H^k(F\cap B(y, r))\geq \delta  r^k
\end{cases}
\end{equation}
and
\begin{equation}\label{eq:upperbound}
\begin{cases}
\H^k(F\cap B(x, 2r)\setminus C_\alpha^{ r}(x))\leq \eps  r^k\\
\H^k(F\cap B(y, 2r)\setminus C_\alpha^{ r}(y))\leq \eps  r^k
\end{cases}
\end{equation}
where $\eps \le \tfrac{\delta}{4}$.
Then:
\begin{enumerate}
    \item[(i)] $$\H^k\left(F\cap C_\alpha^r(x)\cap C_\alpha^r(y)\cap B(x,r)\right)\ge \tfrac{\delta}{2}r^k$$
    and in particular $B(x,r)\cap C_\alpha^{r}(x)\cap C_\alpha^{r}(y)\neq \emptyset$.
    \item[(ii)] $$d(V_x,V_y)\leq C(n,\delta,M,\lambda) r^\alpha$$
    where $C(n,\delta,M,\lambda)=\frac{20^{n+1}2n M\lambda}{\delta\omega_n}$.
\end{enumerate}
\end{lemma}

\begin{proof}
$(i)$ Let us define $A=F\cap C_\alpha^r(x)\cap B(x,r)$. From assumptions \eqref{eq:lowerbound} and \eqref{eq:upperbound} we have that $\H^k(A)\geq (\delta-\eps)r^k$. Moreover since  
\[A\setminus C_\alpha^r(y)\subseteq  F\cap B(y,2r)\setminus C_\alpha^r(y)
\]
from assumption \eqref{eq:upperbound} it follows that $\H^k(A\setminus C_\alpha^r(y))\leq \eps r^k$. Therefore we obtain
\[
(\delta-\eps)r^k\leq \H^k(A)=\H^k(A\cap C_\alpha^r(y))+\H^k(A\setminus C_\alpha^r(y))\leq \H^k(A\cap C_\alpha^r(y)) + \eps r^k.
\]
This implies 
\[
\H^k(A\cap C_\alpha^r(y))\geq (\delta -2\eps)r^k\geq \frac{\delta}{2}r^k,
\]
which concludes.

$(ii)$ We set $C=C(n,\delta,M,\lambda)$ for simplicity. Suppose on the contrary that $\theta:=d(V_x, V_y)>Cr^\alpha$. 
We first claim that there exists an affine $(k-1)$-plane $W$ such that 
\begin{equation}\label{eq:cylinderW}
    C_\alpha^{r}(x)\cap C_\alpha^{r}(y)\subset W(4n\lambda r^{1+\alpha}/\theta).
\end{equation}

If $V_x$ and $V_y$ intersect then the conclusion follows directly from Lemma \ref{lemma:tube}. If on the other hand they do not intersect we consider two points $\tilde x\in V_x,\tilde y\in V_y$ realizing the minimum distance, i.e. such that $\Delta=|\tilde x-\tilde y|=\inf\{|x'-y'|:x'\in V_x, y' \in V_y\}$. Then $\tilde x-\tilde y\perp V_x,V_y$ and we translate both $V_x$ and $V_y$ in direction $\tilde x-\tilde y$ obtaining two $k$-planes $\tilde V_x$ and $\tilde V_y$ both containing the midpoint $\tfrac{\tilde x+\tilde y}{2}$. Since by point $(i)$ $C_\alpha^r(x)\cap C_\alpha^r(y)\neq \emptyset$ we necessarily have $\Delta\leq 2\lambda r^{1+\alpha} $ and therefore $C_\alpha^{r}(x)=B(V_x, \lambda r^{1+\alpha}) \subseteq B(\tilde V_x,2\lambda r^{1+\alpha})$ and similarly for $y$. In conclusion we obtain 
\[B(V_x, \lambda r^{1+\alpha})\cap B(V_y, \lambda r^{1+\alpha})\subseteq B(\tilde V_x,2\lambda r^{1+\alpha})\cap B(\tilde V_y,2\lambda r^{1+\alpha})
\]
where $\tilde V_x$ and $\tilde V_y$ now intersect and we conclude again by Lemma \ref{lemma:tube}. 

Define now $\eta:=4n\lambda r^{1+\alpha}/\theta$ and $F_\eta:=F\cap B(x,r)\cap B(W, \eta)$. First of all we can assume that $\eta\leq r$, for otherwise we obtain
\[
\frac{4n\lambda r^{1+\alpha}}{\theta}>r\implies \theta <4n\lambda r^{\alpha}
\]
which contradicts the assumption we made that $\theta >Cr^\alpha$.

We now claim that 
\begin{equation}\label{eq:claim}
    \H^k(F_\eta)< \frac{20^{n+1}nM\lambda}{C\omega_n}r^k.
\end{equation} 
Indeed  $F_\eta\subseteq \bigcup_{x\in F_\eta} B(x,\eta)$
and by Vitali Covering Theorem we extract a disjoint finite subfamily $\big(B(x_i,\eta)\big)_{i=1}^N$ such that
\[
F_\eta\subseteq \bigcup_{i=1}^N B(x_i,5\eta).
\]
Then on one hand from the uniform upper density assumption \eqref{eq:upperhyp} we have
\begin{equation}\label{eq:Hkbound}
\H^k(F_\eta)\leq \sum_{i=1}^N \H^k(F_\eta\cap B(x_i,5\eta))\leq N 5^kM\eta^k.
\end{equation}
On the other hand we can bound $N$ from above: indeed
\[
\bigcup_{i=1}^N B(x_i,\eta)\subseteq W(2\eta)\cap B(x,r+\eta)\subseteq W(2\eta)\cap B(x,2r).
\]
and therefore since the balls are disjoint
\[
N \L^n(B(x,\eta))\leq \L^n(W(2\eta)\cap B(x,3r))\leq (4r)^{k-1}(4\eta)^{n-k+1}
\]
which implies 
\begin{equation}\label{eq:Nbound}
N\leq \frac{4^n}{\omega_n}\left(\frac{r}{\eta}\right)^{k-1}.
\end{equation}
Putting together \eqref{eq:Hkbound}, \eqref{eq:Nbound}, the definition of $\eta$ and the assumption that $\theta>Cr^\alpha$ we obtain 
\begin{align*}
\H^k(F_\eta)&\leq N 5^k M\eta^k\leq \frac{4^n}{\omega_n}\left(\frac{r}{\eta}\right)^{k-1} 5^k M\eta^k= \frac{4^n 5^k M}{\omega_n}r^k\frac{\eta}{r}\\
&=\frac{4^n 5^k M }{\omega_n}r^k 4n\lambda\frac{r^\alpha}{\theta}<\frac{4^{n+1} 5^k  Mn\lambda }{C \omega_n}r^k\leq \frac{20^{n+1} Mn\lambda}{C\omega_n}r^k
\end{align*}
which is claim \eqref{eq:claim}.

On the other hand by point $(i)$ and \eqref{eq:cylinderW} we have the lower bound $\H^k(F_\eta)\geq \frac{\delta}{2}r^k$ which together with the previous upper bound implies
\[
\frac{\delta}{2}r^k\leq \H^k(F_\eta)< \frac{20^{n+1} n M\lambda}{C\omega_n} r^k 
\]
which is a contradiction by the definition of $C$.
\end{proof}

\begin{remark} We will apply the above lemma twice: the first time with $r=|x-y|$ to obtain the H\"{o}lder continuity of the planes in the proof of Theorem \ref{thm:C1alpharectifgen} and Proposition \ref{thm:C1alpharectifquant}. The second time in the proof of Theorem \ref{thm:C1alpharot} and Proposition \ref{thm:C1alpharectifCylindersquant} with $x=y$ to obtain the rate of tilting of $V_{x,r}$ from one scale $r_j$ to the smaller one $r_{j+1}$ in the case of the rotating planes, and to conclude that the planes stabilize. As a consequence $\H^k$-a.e. point has an approximate tangent paraboloid as per \eqref{eq:parabhyp} and we can conclude by applying Theorem \ref{thm:C1alpharectifgen}.
\end{remark}

The following is an elementary lemma used in the proof of Proposition \ref{thm:C1alpharectifquant} and Proposition \ref{thm:C1alpharectifCylindersquant}.

\begin{lemma}\label{Lemma:Elem}
 Fix $\lambda>0$ and $\alpha\in (0, 1]$. Let $a,b, w \ge 0$. Suppose that 
\[
a^2+b^2=w^2,\quad
a\ge \lambda b^{1+\alpha} \quad\text{ and }\quad a\le \frac{1}{\lambda^{1/\alpha}}.
\]
Then $a\ge \frac{\lambda}{2}w^{1+\alpha}$.
\end{lemma}
\begin{proof} Using the hypotheses we obtain that
\[
    w^2\le a^2+\lambda^{-\frac{2}{1+\alpha}}a^{\frac{2}{1+\alpha}}
    =a^{\frac{2}{1+\alpha}}\left(a^{\frac{2\alpha}{1+\alpha}}+\lambda^{-\frac{2}{1+\alpha}}\right)
    \le a^{\frac{2}{1+\alpha}}2\lambda^{-\frac{2}{1+\alpha}}.
\]
Hence we have that 
\[
a\ge \frac{\lambda}{2^{(1+\alpha)/2}}w^{1+\alpha}\geq \frac{\lambda}{2}w^{1+\alpha}.
\]
\end{proof}
\subsection{Proof of Proposition \ref{thm:C1alpharectifquant}}
We put together Lemma \ref{lemma:key} and Lemma \ref{lemma:geom} to obtain the proof of Proposition \ref{thm:C1alpharectifquant}. The estimate \eqref{eq:ball} below is similar to the conclusion of \cite[Lemma~3.6]{AnzSer}, but since in our case the proof is a short application of Lemma \ref{Lemma:Elem} we write down the details.
\begin{proof}[Proof of Proposition \ref{thm:C1alpharectifquant}]
Since 
\[
B(x,r)\cap Q_\alpha(x,V_x,\lambda)\subset C_\alpha^r(x)
\]
we also have that $\H^k(F\cap B(x,r)\setminus C_\alpha^r(x))\leq\eps  r^k$ whenever \eqref{eq:parabhypquant} holds. 
For any pair of points $x,y\in F'$ we apply Lemma \ref{lemma:key} with $r=|x-y|$ to obtain that the map $x\mapsto V_x$ is $\alpha$-H\"{o}lder when restricted to $F'$, that is
\[
d(V_x,V_y)\leq C|x-y|^\alpha \quad\text{for every $x,y\in F'$}
\]
where $C=C(n,\delta,M,\lambda)$ is the constant of Lemma \ref{lemma:key}. Observe that $|x-y|\leq \mathrm{diam}(F')\leq r_1\leq \tfrac{r_0}{5}$ so that in Lemma \ref{lemma:key} assumption \eqref{eq:upperhyp} holds.

We now claim that for a sufficiently large $\lambda'>\lambda$ (to be chosen later) we have the stronger condition $F'\setminus Q_\alpha(x,V_x,\lambda')=\emptyset$. Indeed suppose by contradiction there are $x,y\in F'$ such that $y\in F'\setminus Q_\alpha(x,V_x,\lambda')$, and set $r=2|x-y|$. We claim that (see Figure \ref{fig:disjoint}) \begin{equation}\label{eq:ball}
    B(y,\tfrac{r}{4})\cap C_\alpha^{r/4} (y)\subset B(x,r)\setminus Q_\alpha(x,V_x,\lambda).
\end{equation}
Indeed, using Lemma \ref{Lemma:Elem} with $w=|y-x|$, $a=|P_{V_x}(y-x)|$, $b=|P_{V_x^\perp}(y-x)|$ we have that for every $z\in C_\alpha^{r/4}(y)$
\begin{align*}
    |P_{V_x^\perp}(z-x)|&=|P_{V_x^\perp}(y-x)+P_{V_y^\perp}(z-y)+(P_{V_x^\perp}-P_{V_y^\perp})(z-y)|\\
    & \geq |P_{V_x^\perp}(y-x)|-|P_{V_y^\perp}(z-y)|-|(P_{V_x^\perp}-P_{V_y^\perp})(z-y)|\\
    & \geq \frac{\lambda'}{2}|x-y|^{1+\alpha}-\lambda |P_{V_y}(z-y)|^{1+\alpha}-C|x-y|^\alpha |z-y|\\
    & \geq \frac{\lambda'}{2} \left(\frac{r}{2}\right)^{1+\alpha}-\lambda \left(\frac{r}{4}\right)^{1+\alpha}-2C\left(\frac{r}{4}\right)^{1+\alpha}\\
    & > \lambda r^{1+\alpha}\geq \lambda|P_{V_x}(z-x)|^{1+\alpha}
\end{align*}
where we have chosen $\lambda'=2C+(1+4^{1+\alpha})\lambda> \frac{2C+(1+4^{1+\alpha})\lambda}{2^\alpha}$. 

Thus we deduce \eqref{eq:ball} and hence obtain that
\[
(\delta-\eps) \left(\frac{r}{4}\right)^k\leq \H^k(F\cap B(y,\tfrac{r}{4})\cap C_\alpha^{r/4}(y))\leq\H^k(F\cap B(x,r)\setminus Q_\alpha(x,V_x,\lambda))\leq \eps r^k
\]
which is a contradiction since $\eps<\tfrac{1}{4^k+1}\delta$. We have thus proved that 
\[F'\setminus Q_\alpha(x,V_x,\lambda')=\emptyset \quad\text{ for every $x\in F'$.}\]
We can thus apply Lemma \ref{lemma:geom} to the set $F'$ to obtain the desired conclusion.
\end{proof}

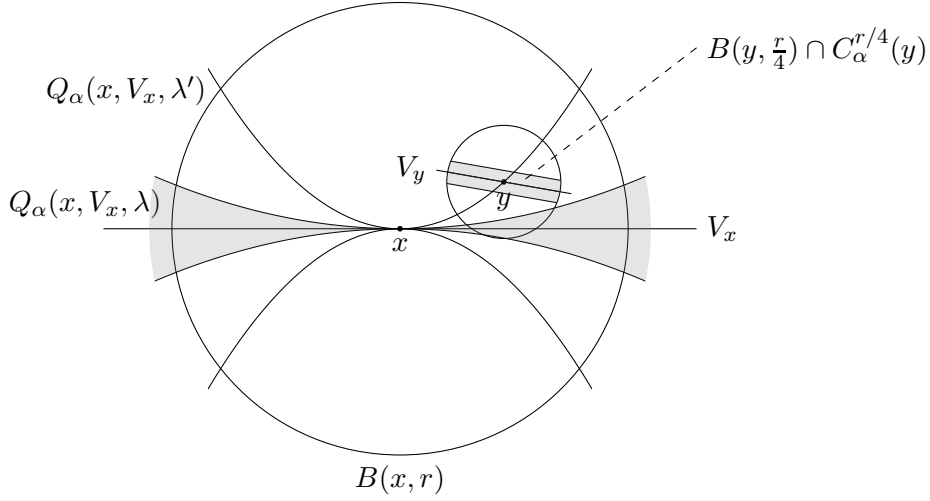
\begin{figure}
\centering
\begin{tikzpicture}[scale=3]
\draw[fill] (0,0) circle [radius=0.01] node[below] {$x$};
\draw[fill] (0.455,0.207) circle [radius=0.01] node[below] {$y$} ;
\draw (0,0) circle [radius=1];
\node[below] at (0,-1) {$B(x,r)$};
\draw (-1.3,0)--(1.3,0) node[right] {$V_x$};
\draw (0.455,0.207) circle [radius=0.25];

\begin{scope}[shift={(0.455,0.207)},rotate=-10]
\draw (0.3,0) -- (-0.3,0) node[left] {$V_y$};
\end{scope}

\begin{scope}[shift={(0.455,0.207)},rotate=-10]
\clip (0,0) circle [radius=0.25]; 
\draw (-1,0) -- (1,0);
\draw (-1,-0.05) rectangle (1,0.05);
\fill[gray,opacity=0.2] (-1,-0.05) rectangle (1,0.05);
\end{scope}

\node[above left] at (-1,0) {$Q_\alpha(x,V_x,\lambda)$};
\node[above] at (-1.2,0.5) {$Q_\alpha(x,V_x,\lambda')$};
\draw[dashed] (0.55,0.22) -- (1.3,0.8) node[right] {$B(y,\tfrac{r}{4})\cap C_\alpha^{r/4}(y)$};

\begin{scope}
\clip (0,0) circle [radius=1.1];
\draw [smooth,samples=100,domain=-1.2:1.2] plot(\x,{\x*\x});
\draw [smooth,samples=100,domain=-1.2:1.2] plot(\x,{-\x*\x});
\draw [smooth,samples=100,domain=-1.2:1.2] plot(\x,{0.2*\x*\x});
\draw [smooth,samples=100,domain=-1.2:1.2] plot(\x,{-0.2*\x*\x});
\fill [gray,opacity=0.2, domain=-1.2:1.2, variable=\x]
      (-1.2, 0)
      -- plot ({\x}, {0.2*\x*\x})
      -- (1.2, 0)
      -- cycle;
\fill [gray,opacity=0.2, domain=-1.2:1.2, variable=\x]
      (-1.2, 0)
      -- plot ({\x}, {-0.2*\x*\x})
      -- (1.2, 0)
      -- cycle;
\end{scope}
\end{tikzpicture}
\caption{Reference for the proof of Proposition \ref{thm:C1alpharectifquant}.}\label{fig:disjoint}
\end{figure}

\subsection{Proof of Theorem \ref{thm:C1alpharectifgen}}
As a consequence of a standard decomposition argument and Proposition \ref{thm:C1alpharectifquant} we finally obtain the first main result.
\begin{proof}[Proof of Theorem \ref{thm:C1alpharectifgen}]
Observe that by Theorem \ref{thm:rectifgen} we obtain that the set $E$ is $\H^k$-rectifiable. Therefore for $\H^k$-a.e. $x\in E$ it holds that $\Theta^{k}(E,x)=1$ \cite[Theorem~16.2]{Mat95}.
 
Let $\lambda,r_0>0$ be arbitrary but fixed and define $E_{\lambda,r_0}$ as the set of all $x\in E$ such that for every $0<r\leq r_0$
\[
\begin{cases}
    \H^k(E\cap B(x,r))\geq \frac12 (2r)^k\\
    \H^k(E\cap B(x,r))\leq 2(2r)^k\\
    \H^k(E\cap B(x,r)\setminus Q_\alpha(x,V_x,\lambda))\leq \frac{1}{8(4^k+1)} r^k
\end{cases}
\]
Clearly $\H^k(E\setminus\bigcup_{\lambda,r_0}E_{\lambda,r_0})=0$, and moreover the union can be taken among countably many values for the parameters. Hence it suffices to show that each $E_{\lambda,r_0}$ is $C^{1,\alpha}$ rectifiable. 

We further decompose $E_{\lambda,r_0}$ and define
\[
E_{\lambda,r_0}^t:=\{x\in E_{\lambda,r_0}: \H^k(E_{\lambda,r_0}\cap B(x,r))\geq \tfrac12 \H^k(E\cap B(x,r))\text{ for every $0<r\leq t$}\}.
\]
Again $\H^k\left(E_{\lambda,r_0}\setminus\bigcup_t E_{\lambda,r_0}^t\right)=0$ because by a standard density result the relative density of a subset is $1$ almost everywhere on the subset  \cite[Corollary~2.14]{Mat95}, and the union can be taken among countably many values of $t$. It is thus sufficient to prove that each $E_{\lambda,r_0}^t$ is $C^{1,\alpha}$ rectifiable.

Now we can directly apply Proposition \ref{thm:C1alpharectifquant} with $M=2^{k+1},\delta=2^{k-1}$ and $F'=E_{\lambda,r_0}^t$, $F=E_{\lambda,r_0}$ to obtain the conclusion.
\end{proof}

\section{Rotating planes}
We now pass to the proof of Proposition \ref{thm:C1alpharectifCylindersquant}, Theorem \ref{thm:C1alpharot} and Corollary \ref{cor:ghi}. The idea is the following: under the assumptions of the proposition, for a fixed $x\in E$ the planes $V_{x,r}$ actually have to converge with a precise rate as $r\to 0$, so that in fact at $\H^k$-a.e. $x\in E$ estimate \eqref{eq:parabhypquant} holds for a certain $V_x$, and we can apply Proposition \ref{thm:C1alpharectifquant}. Theorem \ref{thm:C1alpharot} will follow as a consequence of Proposition \ref{thm:C1alpharectifCylindersquant}.

In the following lemma we prove that under the conclusions of Lemma \ref{lemma:key} we can prove the convergence of $V_{x,r}$ to a certain $V_x$ with a precise rate $r^\alpha$. Let us fix as usual a sequence of radii $r_j=r_0\rho^j$ for some $r_0>0$ and $0<\rho<1$. For simplicity we consider the case where $x=0$.

\begin{lemma}\label{lemma:rot}
Consider a sequence of affine $k$-planes $V_j\in A(n,k)$ and suppose that there exist constants $\lambda,C>0$ such that for every $j\in \N$
\begin{equation}\label{eq:intersection}
B(V_j, \lambda r_j^{1+\alpha})\cap B(V_{j+1}, \lambda r_j^{1+\alpha})\cap B(0,r_j)\neq \emptyset
\end{equation}
and
\begin{equation}\label{eq:thetabound}
    \theta_j:=d(V_j,V_{j+1})\leq C r_j^\alpha.
\end{equation}
Then there exists a linear $k$-plane $V_\infty\in G(n,k)$ such that for every $j\in \N$
\begin{equation}\label{eq:Vinfty}
    B(V_j, \lambda r_j^{1+\alpha})\cap B(0,r_j)\subseteq B(V_\infty, \lambda'' r_j^{1+\alpha})
\end{equation}
    where $\lambda''=\lambda+C+\frac{2\lambda+C}{1-\rho^{1+\alpha}}$.

    

\end{lemma}

\begin{figure}
\centering
    \begin{tikzpicture}[scale=3]
    \draw[fill] (0,0) circle [radius=0.01] node[below] {$0$};
    \begin{scope}
    \draw[thick] (-1,0.5) -- (1,0.5) node[right] {$V_j$};
    \draw (-1,0.35) -- (1,0.35) node[below right] {$W_{j}$};
    \draw[dotted] (0,0.5) -- (0.5,1) node[right] {$x_j$};
    \draw[fill] (-0.7,0.35) circle [radius=0.01] node[below] {$y_j$};
    
    \begin{scope}
    \clip (0,0) circle [radius=1];
    
    \draw[<->,gray] (0,0) -- (-0.8,-0.6);
    \node[below] at (-0.6,-0.5) {$r_j$};

    \draw[thick] (0,0) circle [radius=1];
    
    \fill[gray,opacity=0.2] (-1,0.3) rectangle (1,0.7);

    \draw[dashed] (0,0) -- (0,0.5);
    \draw[fill] (0,0.5) circle [radius=0.01];
    
    \end{scope}
    
    \end{scope}
    \begin{scope}[rotate=-8,scale=0.4]
    
    \draw[thick] (-1,0.5) -- (1,0.5) node[below right] {$V_{j+1}$};
    \draw (-2.5,0.63) -- (2.5,0.63) node[below right] {$W_{j+1}$};
    \draw[<->,gray] (0,0) -- (0.8,-0.6);
    \node[right] at (0.4,-0.3) {$r_{j+1}$};
    \draw[dotted] (0,0.5) -- (-3,-0.5) node[below] {$x_{j+1}$};
    
    \begin{scope}
    \clip (0,0) circle [radius=1];

    \draw[thick] (0,0) circle [radius=1];
    
    \fill[gray,opacity=0.2] (-1,0.3) rectangle (1,0.7);
    \draw (-1,0.5) -- (1,0.5);
    
    \draw[dashed] (0,0) -- (0,0.5);
    \draw[fill] (0,0.5) circle [radius=0.025];
    
    \end{scope}
    \end{scope}
    
    
    
    
    
    
    
    
    
    
    
    
    
    
    
    
    

    \end{tikzpicture}
    \caption{Reference figure for the proof of Lemma \ref{lemma:rot}. 
    }
    \label{fig:Vinfty}
\end{figure}
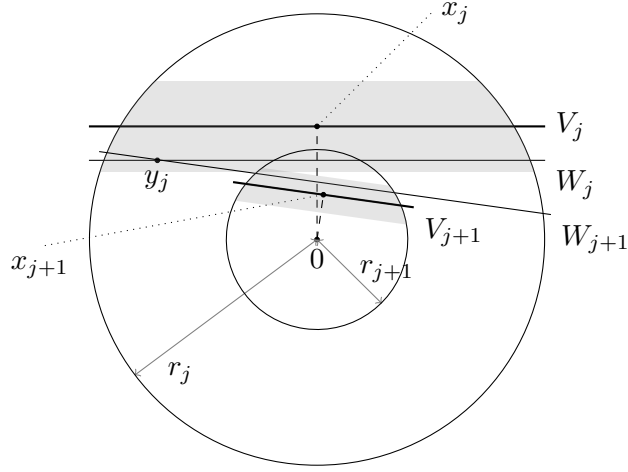

\begin{proof}
$(i)$ For every $j\in \N$ we call $x_j$ the unique point in $ V_j$ nearest to the origin. We first claim that $|x_j-x_{j+1}|\leq (2\lambda + C)r_j^{1+\alpha}$. 

Fix $j\in \N$. By \eqref{eq:intersection} we can find two affine $k$-planes $W_j=\tau_j+V_j$ and $W_{j+1}=\tau_{j+1}+V_{j+1}$ such that $W_j\cap W_{j+1}\cap B(0,r_j)\neq \emptyset$, where $|\tau_j|,|\tau_{j+1}|\leq \lambda r_j^{1+\alpha}$. Let $x'_j\in W_j$ and $ x'_{j+1}\in W_{j+1}$ be the respective points of minimum distance from the origin. Then $|x'_j-x_j|\leq \lambda r_j^{1+\alpha}$ and $| x'_{j+1}-x_{j+1}|\leq \lambda r_j^{1+\alpha}$. Let also $y_j$ be any point in $W_j\cap W_{j+1}\cap B(0,r_j)$, so that $\tilde V_j:=W_j-y_j$ and $\tilde V_{j+1}:=W_{j+1}-y_j$ are $k$-planes containing the origin, and the projections $P_{\tilde V_j}$ and $P_{\tilde V_{j+1}}$ are linear maps. Then we obtain
\begin{align*}
|x'_j- x'_{j+1}|&=|( x'_j-y_j)-(x'_{j+1}-y_j)|=|P_{\tilde V_j}(-y_j)-P_{\tilde V_{j+1}}(-y_j)|\\
&=|(P_{\tilde V_j}-P_{\tilde V_{j+1}})(-y_j)|\leq C r_j^{1+\alpha}
\end{align*}
and thus 
\[
|x_j-x_{j+1}|\leq 2\lambda r_j^{1+\alpha}+|\tilde x_j-\tilde x_{j+1}|\leq (2\lambda + C)r_j^{1+\alpha}
\]
so that the claim is proved.

In particular, since the points $x_j$ are converging to the origin, we obtain
\begin{equation}\label{eq:xj}
|x_j|\leq \sum_{j=i}^\infty |x_i-x_{i+1}|\leq \sum_{i=j}^\infty (2\lambda+ C)r_i^{1+\alpha}= (2\lambda+ C)\frac{r_j^{1+\alpha}}{1-\rho^{1+\alpha}}.
\end{equation}
Finally we prove \eqref{eq:Vinfty}. Given any $z\in B(V_j,\lambda r_j^{1+\alpha})$, by triangle inequality and using \eqref{eq:intersection}, \eqref{eq:thetabound} and \eqref{eq:xj} we have that
\begin{align*}
    |P_{V_\infty^\perp}z|&\leq|P_{V_j^\perp}(z-x_j)|+|(P_{V_\infty^\perp}-P_{V_j^\perp})(z-x_j)|+|P_{V_\infty^\perp}(x_j)|\\
    &\leq \lambda |P_{V_j}(z-x_j)|^{1+\alpha}+Cr_j^{\alpha}|z-x_j|+|x_j|\\
    &\leq \left(\lambda+C+\frac{2\lambda +C}{1-\rho^{1+\alpha}}\right) r_j^{1+\alpha}.
\end{align*}
\end{proof}

\subsection{Proof of Proposition \ref{thm:C1alpharectifCylindersquant} and Theorem \ref{thm:C1alpharot}}

\begin{proof}[Proof of Proposition \ref{thm:C1alpharectifCylindersquant}] We essentially want to replace (\ref{eq:parabhypCyl}) with an approximate paraboloid estimate that allows us to pass to Proposition \ref{thm:C1alpharectifquant}.\\[3pt]
Let $x\in F'$. 
Taking $y=x$ in Lemma \ref{lemma:key} and $V_x=V_j:=V_{x,r_j}$, $V_y=V_{j+1}:=V_{x,r_{j+1}}$ and $r=r_{j+1}$, we obtain the conditions (\ref{eq:intersection}) and (\ref{eq:thetabound}) of Lemma \ref{lemma:rot} and hence an affine $k$-plane $V_{x,\infty}$ such that 
\[
 B(V_j, \lambda r_j^{1+\alpha})\cap B(x,r_j)\subseteq B(V_{x,\infty}, \lambda'' r_j^{1+\alpha}),
\]
where $\lambda''=\lambda+C+\frac{2\lambda+C}{1-\rho^{1+\alpha}}$ (as in Lemma \ref{lemma:rot}).
By Lemma \ref{lemma:parabcyl} it follows that 
\begin{equation*}
\H^k(F\cap B(x,r)\setminus Q_\alpha(x,V_{x,\infty},4^{1+\alpha}\lambda''))\leq \frac\eps{1-2^{-k}} r^k.
\end{equation*}
 This reduces to the setting of Proposition \ref{thm:C1alpharectifquant} and hence the result follows.
\end{proof}

\begin{proof}[Proof of Theorem \ref{thm:C1alpharot}]
 Since $\H^k(E)<\infty$ by \eqref{eq:Mat6.2} for $\H^k$-a.e. point $x\in E$ we have $\Theta^{*k}(E,x)\leq 1$. Given $r_0,\delta,\lambda>0$, let $E_{r_0,\lambda,\delta}$ denote the set of points in $E$ such that for every $0<r<r_0$ the following conditions hold:
\[
\begin{cases}
\H^k(E\cap B(x,r))\leq 2(2r)^k.\\
\H^k(E\cap B(x,r))\geq\delta (2r)^k\\
\H^k(E\cap B(x,r)\setminus B(V_{x,r},\lambda r^{1+\alpha}))\leq \frac{1}{4(4^k+1)} \delta r^k 
\end{cases}.
\]
 Then clearly, $\H^k(E\setminus \bigcup_{r_0,\lambda,\delta}E_{r_0,\lambda,\delta})=0$  and the union can be taken among countably many positive values for all the parameters, hence it is enough to establish claim for each of the sets $E_{r_0,\lambda,\delta}$.

We further decompose the sets: for any $t>0$ define $E_{r_0,\lambda,\delta}^t$ as the set of those points $x\in E_{r_0,\lambda,\delta}$ such that
\[
\H^k(E_{r_0,\lambda,\delta}\cap B(x,r))\geq \tfrac12 \H^k(E\cap B(x,r)) \text{ for every $0<r\leq t$}.
\]
As in the proof of Theorem \ref{thm:C1alpharectifgen} it is sufficient to conclude that $E_{r_0,\lambda,\delta}^t$ is $C^{1,\alpha}$ rectifiable, but this reduces to Proposition \ref{thm:C1alpharectifCylindersquant} applied to the sets $F'=E_{r_0,\lambda,\delta}^t$, $F=E_{r_0,\lambda,\delta}$.
\end{proof}

\begin{remark}\label{rmk:moduli}
We briefly go through the modifications needed to obtain the more general version of Theorems \ref{thm:C1alpharectifgen} and \ref{thm:C1alpharot} stated in Remark \ref{rmk:intromoduli}. First of all Whitney's extension theorem holds for general moduli of continuity \cite[VI.4.6]{Stein}. An analogous proof of Lemma \ref{lemma:key} gives that $d(V_x,V_y)\leq C \lambda(r)\leq C\omega(r)$ whenever $|x-y|\leq r$. In particular for a fixed $x$ we have $d(V_{j},V_{j+1})\leq C\lambda(r_j)\leq C\omega(r_j)$ and thus the analogue of Theorem \ref{thm:C1alpharectifgen} follows.

Regarding the case of rotating planes, a modification of Lemma \ref{lemma:rot} then gives that since $\sum_j \lambda(r_j)<\infty$ there exists a limit plane $V_\infty$ and $d(V_\infty,V_{m})\leq C\omega_m$. Finally $B(V_{j}, \lambda(r_j)r_j)\subseteq V_\infty(C\omega(r_j)r_j)$ and therefore the application of the above generalization of Theorem \ref{thm:C1alpharectifgen} gives the conclusion.
\end{remark}

\subsection{Proof of Corollary \ref{cor:ghi}}

\begin{proof}[Proof of Corollary \ref{cor:ghi}]

First observe that by Lemma \ref{lemma:betap}, for every point $x\in E$ where $\Theta^{*k}(E,x)<\infty$ and where
\eqref{eq:betap} holds we have that $\beta_2(x,r)\leq C r^{\gamma}$ for some constants $C,\gamma>0$ and for all $r>0$ sufficiently small. Therefore $\int_0^1\beta_2(x,r)^2\frac{dr}{r}<\infty$ for all points where $\Theta^{*,k}(E,x)<\infty$ and where \eqref{eq:betap} holds, which is true for $\H^k$-a.e. $x\in E$. Thus we can apply Theorem \ref{thm:QuantRectifGen} to conclude that $E$ is $\H^k$-rectifiable, and in particular $\Theta_*^k(E,x)>0$ for $\H^k$-a.e. $x\in E$.

Fix now $x\in E$ where $\Theta_*^k(E,x)>0$ and where \eqref{eq:betap} holds for some $p$. If $p=\infty$ we conclude immediately by Theorem \ref{thm:C1alpharot}. Suppose then $1\leq p<\infty$. For every $r>0$ consider a minimizer $V_{x,r}$ of $\beta_p(x,r)$. Then
\[
\beta_p(x,r)^p \ge\frac{1}{r^{k}}\H^k\left(E\cap B(x,r)\setminus B(V_{x,r}, \lambda r^{1+\alpha})\right)(\lambda r^{\alpha})^p
\]
thus using \eqref{eq:betap} and taking $\lambda$ 
sufficiently big we obtain
\begin{align*}
\limsup_{r\to 0}\frac{1}{r^{k}}\H^k\left(E\cap B(x,r)\setminus B(V_{x,r}, \lambda r^{1+\alpha})\right)&\le \limsup_{r\to 0}\frac{1}{r^k}\frac{\beta_p(x,r)^p}{\lambda^p r^{p\alpha}}\\
&< (1-2^{-k})\eps_0(k)\Theta_*^{k}(E,x).
\end{align*}
Now we conclude using Theorem \ref{thm:C1alpharot}.
\end{proof}

 



\small
\bibliographystyle{alpha}
\bibliography{C1alpha}

\end{document}